\documentclass[11pt]{article}

\usepackage{amsmath,amssymb,amsthm}
\usepackage{enumitem}
\usepackage{booktabs}
\usepackage{graphicx}
\usepackage[hidelinks]{hyperref}
\usepackage{tikz}
\usetikzlibrary{arrows.meta,positioning}

\newtheorem{theorem}{Theorem}[section]
\newtheorem{proposition}[theorem]{Proposition}

\theoremstyle{definition}
\newtheorem{definition}[theorem]{Definition}

\theoremstyle{remark}
\newtheorem{remark}[theorem]{Remark}

\newcommand{\E}{\mathbb{E}}
\newcommand{\Prob}{\mathbb{P}}
\newcommand{\R}{\mathbb{R}}

\newcommand{\dinf}{d_\infty}

\title{Approximation of Probability Geometries\\ via Convex Families of Coordinate Charts}
\author{Manuela-Simona Cojocea}
\date{}

\begin{document}

\maketitle

\begin{abstract}
Within probability-coordinate geometry, a continuous
distribution function restricted to an interval acts as a \emph{probability
coordinate chart}, and averaging performed in probability coordinates defines
the \emph{probability barycenter} $b_G(X)=G^{-1}(\E[G(X)])$ together with the
associated Kolmogorov moments. The class of admissible charts is infinite
dimensional. This paper develops a finite-dimensional approximation theory for
it. Fixing a base chart $F$, every chart on the same interval factorizes
uniquely as $G=K\circ F$, where $K$ is a strictly increasing deformation map of
the unit interval; the choice of a probability geometry is thereby reduced to
the choice of a deformation of the canonical probability scale. We prove that
monotone piecewise-linear deformations admit an exact representation as convex
combinations of elementary ramp activations, so that probability geometries are
parametrised by allocation vectors in a simplex, whose relative interior
corresponds exactly to admissible charts. Piecewise-linear geometries are dense
in the admissible class, and two constructive schemes are analysed: ramp
interpolation and Bernstein approximation. The Bernstein scheme is shown to be
itself a ramp-dictionary representation with smooth beta atoms---the
distribution functions of uniform order statistics---and it regularizes:
Bernstein approximants of arbitrary monotone deformations are automatically
admissible. Uniform approximation of deformations transfers to the geometric
functionals considered here: probability barycenters and initial
Kolmogorov moments are stable under $\dinf$-perturbations of the deformation
map, with explicit bounds. For deformed intrinsic charts the barycenter is a
quantile whose level is determined by the deformation alone, which links the
present construction to chart-based quantile estimators. All
numerical illustrations are deterministic and reproducible.
\end{abstract}

\section{Introduction}\label{sec:intro}

The probability-coordinate approach to expectation
\cite{Cojocea2026a,Cojocea2026b} gives a geometric interpretation in which a
continuous distribution function,
restricted to an interval where it is strictly increasing, is read as a
\emph{probability coordinate chart}: a coordinate map transporting value space
into the unit interval. Averaging performed in probability coordinates and
pulled back through the inverse chart defines the \emph{probability barycenter}
\[
b_G(X)=G^{-1}\!\bigl(\E[G(X)]\bigr),
\]
a Kolmogorov mean \cite{Kolmogorov1930,Nagumo1930} whose existence requires no
integrability of $X$, together with the initial Kolmogorov moments
$M_r^{(G)}(X)=G^{-1}(\E[(G(X))^r])$. The geometric foundations and limit
theorems are developed in \cite{Cojocea2026a}, while statistical inference,
organized around the choice of the chart, is studied in \cite{Cojocea2026b}.
Related applications treat quantile estimation for stable laws
\cite{CojoceaPele2026} and exponential families transported through charts
\cite{Cojocea2026c}.

In all of these developments the chart $G$ is an infinite-dimensional
parameter: any continuous distribution function strictly increasing on the
relevant interval qualifies. A central practical question, raised explicitly in
the chart-selection problem of \cite{Cojocea2026b}, is therefore whether the
class of probability geometries admits tractable finite-dimensional
parametrisations. The present paper answers this question through a convex
approximation theory.

The starting point is a factorization. Fix a base chart $F$ on an interval
$I$---for instance the intrinsic chart $F=F_X$ of the random variable under
study, or a benchmark chart in the sense of \cite{Cojocea2026a}. Every chart
$G$ on $I$ then factorizes uniquely as
\[
G = K \circ F,
\]
where $K:[0,1]\to[0,1]$ is a strictly increasing continuous map fixing the
endpoints (Proposition~\ref{prop:correspondence}). The map $F$ carries the
probabilistic structure; the \emph{deformation map} $K$ redistributes geometric
resolution across the canonical probability scale. Regions where $K$ varies
rapidly receive greater geometric emphasis; regions where $K$ varies slowly are
geometrically compressed. The choice of a probability geometry is thereby
reduced to the choice of a deformation of the canonical probability space, and
the approximation problem for charts is reduced to an approximation problem for
deformation maps on $[0,1]$.

The main results are as follows.

\begin{enumerate}[label=(\roman*)]
\item \emph{Exact correspondence} (Section~\ref{sec:deformation}): relative to
a base chart, probability geometries on an interval are in bijection with
admissible deformation maps, and convex combinations of deformation maps are
again deformation maps, so that finite dictionaries generate simplices of
geometries.
\item \emph{Ramp decomposition} (Section~\ref{sec:ramp}): every non-decreasing
piecewise-linear deformation is an exact convex combination of elementary ramp
activations localised on the partition intervals; the weights
$\lambda_k=K(t_k)-K(t_{k-1})$ describe how deformation is allocated across
probability space, and the admissible (strictly increasing) deformations
correspond exactly to the relative interior of the simplex.
\item \emph{Density and constructive schemes} (Section~\ref{sec:approx}):
piecewise-linear geometries are dense in the admissible class; Bernstein
approximants provide a smooth constructive scheme which is itself a
ramp-dictionary representation whose atoms are the beta distribution functions
of uniform order statistics, and which automatically produces admissible
deformations.
\item \emph{Transfer to probability geometry}
(Section~\ref{sec:transfer}): uniform approximation of deformation maps
transfers to uniform approximation of charts, and further to stability of the
probability barycenters and initial Kolmogorov moments of
\cite{Cojocea2026a,Cojocea2026b}, with explicit quantitative bounds. For
deformed intrinsic charts, the barycenter is a quantile whose level is
determined by the deformation map alone.
\item \emph{Reproducible illustrations} (Sections~\ref{sec:examples}
and~\ref{sec:numerics}): worked ramp geometries for the standard normal
distribution, and a deterministic numerical study of the two constructive schemes on
an exponential-tilt deformation, including the induced barycenter errors.
\end{enumerate}

Throughout, charts, barycenters, and Kolmogorov moments use the terminology and
notation of \cite{Cojocea2026a}, while the inferential reading of estimated
charts follows \cite{Cojocea2026b}.

\section{Deformation maps and probability coordinate charts}
\label{sec:deformation}

\subsection{Probability coordinate charts}

We recall the basic objects of the framework; see \cite{Cojocea2026a} for the
full development.

\begin{definition}[Probability coordinate chart, \cite{Cojocea2026a}]
\label{def:chart}
Let $\widetilde G:\R\to[0,1]$ be a continuous cumulative distribution function
and let $I=(a,b)\subseteq\R$ be an interval such that the restriction
$G=\widetilde G|_I$ is strictly increasing with
\[
\lim_{x\to a^+}G(x)=0,
\qquad
\lim_{x\to b^-}G(x)=1.
\]
Then $G:I\to(0,1)$ is called the \emph{probability coordinate chart} induced by
$\widetilde G$.
\end{definition}

A chart is thus a continuous strictly increasing bijection from $I$ onto
$(0,1)$ arising from a distribution function. For an $I$-valued random
variable $X$, the coordinate $G(X)\in(0,1)$ is bounded, so $\E[G(X)]$ always
exists, and the \emph{probability barycenter} (Kolmogorov expectation induced
by $G$)
\[
b_G(X)=G^{-1}\!\bigl(\E[G(X)]\bigr)
\]
is well defined for arbitrary $X$, without integrability assumptions
\cite{Cojocea2026a}. Similarly, for $r\ge1$ the \emph{initial Kolmogorov
moment} of order $r$ is
\[
M_r^{(G)}(X)=G^{-1}\!\bigl(\E[(G(X))^r]\bigr).
\]
When $X$ has continuous strictly increasing distribution function $F_X$ on
$I$, the chart $G=F_X$ is called \emph{intrinsic}; then $F_X(X)\sim
\mathrm{Uniform}(0,1)$, and the intrinsic coordinates are canonical for the law
of $X$.

\subsection{Deformation maps}

\begin{definition}[Deformation maps]\label{def:deformation}
A \emph{monotone deformation map} is a continuous non-decreasing function
$K:[0,1]\to[0,1]$ with $K(0)=0$ and $K(1)=1$. A monotone deformation map is
called \emph{admissible} if it is strictly increasing. We write $\mathcal K$
for the class of monotone deformation maps and $\mathcal K^\circ\subset
\mathcal K$ for the admissible ones.
\end{definition}

Admissibility is exactly what is needed for composition with a chart to
produce a chart; the strictness requirement corrects the naive expectation that
monotonicity alone suffices, and the distinction is structural for the convex
families studied below: natural dictionaries have non-admissible extreme
points, while their convex interiors are admissible
(Proposition~\ref{prop:interior}).

\begin{proposition}[Charts correspond to admissible deformations]
\label{prop:correspondence}
Let $F$ be a probability coordinate chart on the interval $I$.
\begin{enumerate}[label=(\alph*)]
\item If $K\in\mathcal K^\circ$, then $G=K\circ F$ is a probability coordinate
chart on $I$.
\item Conversely, if $G$ is a probability coordinate chart on $I$, then
$K=G\circ F^{-1}:(0,1)\to(0,1)$ extends uniquely to a map in
$\mathcal K^\circ$, and $G=K\circ F$.
\end{enumerate}
Consequently $K\mapsto K\circ F$ is a bijection between $\mathcal K^\circ$ and
the class of probability coordinate charts on $I$.
\end{proposition}

\begin{proof}
(a) The composition $K\circ F:I\to(0,1)$ is continuous and strictly increasing
as a composition of such maps, and
\[
\lim_{x\to a^+}K(F(x))=K(0)=0,
\qquad
\lim_{x\to b^-}K(F(x))=K(1)=1 ,
\]
using the continuity of $K$ at the endpoints. Moreover $\widetilde
K\circ\widetilde G$ is a continuous distribution function on $\R$ whenever
$\widetilde G$ is one inducing $F$, so $K\circ F$ is induced by a distribution
function and Definition~\ref{def:chart} applies.

(b) Both $F$ and $G$ are continuous strictly increasing bijections from $I$
onto $(0,1)$, so $K=G\circ F^{-1}$ is a continuous strictly increasing
bijection of $(0,1)$ onto itself; hence
$\lim_{u\to0^+}K(u)=0$ and $\lim_{u\to1^-}K(u)=1$, and setting $K(0)=0$,
$K(1)=1$ yields the unique continuous extension, which lies in
$\mathcal K^\circ$. By construction $K\circ F=G$ on $I$. Uniqueness of the
factorization is immediate: $K$ is determined on $(0,1)$ by $K=G\circ F^{-1}$.
\end{proof}

\begin{remark}[Geometry, not transformation]\label{rem:geometry}
The factorization $G=K\circ F$ may superficially resemble a monotone
transformation of the data. The geometric reading is different, and it is the
one used in probability-coordinate geometry: a classical transformation acts in
value space, $x\mapsto g(x)$, modifying the observable; here the observable is
first transported into intrinsic probability coordinates $u=F(x)$, and the
deformation $u\mapsto K(u)$ acts \emph{inside probability space},
reallocating coordinate resolution across probability levels. The pair
$(F,K)$---not the law of $X$ alone---specifies the geometry: a single
distribution supports infinitely many probability geometries, one for each
admissible deformation of its canonical coordinates. This is the
representation-dependence of expectation emphasized in \cite{Cojocea2026a},
localized in the deformation component.
\end{remark}

\subsection{Convex families of deformation maps}

\begin{definition}[Dictionary and convex family]\label{def:dictionary}
A \emph{dictionary} of deformation maps is a finite collection
$\mathcal D=\{K_1,\dots,K_m\}\subset\mathcal K$. The associated convex family
is
\[
K_\lambda=\sum_{j=1}^m\lambda_j K_j,
\qquad
\lambda\in\Delta^{m-1}
=\Bigl\{\lambda\in\R^m:\lambda_j\ge0,\ \textstyle\sum_{j=1}^m\lambda_j=1\Bigr\}.
\]
\end{definition}

\begin{proposition}[Convexity]\label{prop:convexity}
$\mathcal K$ is convex: $K_\lambda\in\mathcal K$ for every dictionary
$\mathcal D\subset\mathcal K$ and every $\lambda\in\Delta^{m-1}$. Moreover, if
$\lambda_j>0$ for at least one index $j$ with $K_j\in\mathcal K^\circ$, then
$K_\lambda\in\mathcal K^\circ$.
\end{proposition}

\begin{proof}
Convex combinations of continuous non-decreasing functions are continuous and
non-decreasing, and
\[
K_\lambda(0)=\sum_j\lambda_jK_j(0)=0,
\qquad
K_\lambda(1)=\sum_j\lambda_jK_j(1)=1 .
\]
For the second claim, if $u<v$ then
$K_\lambda(v)-K_\lambda(u)\ge\lambda_j\bigl(K_j(v)-K_j(u)\bigr)>0$.
\end{proof}

Through Proposition~\ref{prop:correspondence}, each $\lambda$ with
$K_\lambda\in\mathcal K^\circ$ defines a probability coordinate chart
\[
G_\lambda=K_\lambda\circ F,
\]
and the simplex $\Delta^{m-1}$ becomes a finite-dimensional parameter space of
probability geometries anchored at the base chart $F$. Rather than selecting a
single chart, one selects a convex allocation of elementary geometric
components---a formulation directly compatible with the estimated-chart
programme of \cite{Cojocea2026b}, where finite-dimensional chart families are
calibrated on data.

\section{Ramp decomposition of piecewise-linear geometries}
\label{sec:ramp}

\subsection{The ramp dictionary}

\begin{theorem}[Ramp decomposition of monotone piecewise-linear deformations]
\label{thm:ramp}
Let $0=t_0<t_1<\dots<t_n=1$ be a partition of $[0,1]$ and let
$K\in\mathcal K$ be piecewise linear with respect to this partition. For
$k=1,\dots,n$ define the \emph{ramp}
\[
H_k(u)=
\begin{cases}
0, & u\le t_{k-1},\\[4pt]
\dfrac{u-t_{k-1}}{t_k-t_{k-1}}, & t_{k-1}<u<t_k,\\[8pt]
1, & u\ge t_k.
\end{cases}
\]
Then
\[
K(u)=\sum_{k=1}^n\lambda_k H_k(u),
\qquad
\lambda_k=K(t_k)-K(t_{k-1}),
\]
with $\lambda_k\ge0$ and $\sum_{k=1}^n\lambda_k=1$. Thus every monotone
piecewise-linear deformation map lies in the convex hull of the ramp dictionary
$\mathcal H_n=\{H_1,\dots,H_n\}$.
\end{theorem}

\begin{proof}
Monotonicity of $K$ gives $\lambda_k=K(t_k)-K(t_{k-1})\ge0$, and the sum
telescopes to $K(1)-K(0)=1$. Let $u\in[t_{i-1},t_i]$. Then $H_k(u)=1$ for
$k<i$, $H_k(u)=0$ for $k>i$, and $H_i(u)=(u-t_{i-1})/(t_i-t_{i-1})$, so
\[
\sum_{k=1}^n\lambda_kH_k(u)
=\sum_{k=1}^{i-1}\lambda_k+\lambda_i\,\frac{u-t_{i-1}}{t_i-t_{i-1}}
=K(t_{i-1})+\bigl(K(t_i)-K(t_{i-1})\bigr)\frac{u-t_{i-1}}{t_i-t_{i-1}},
\]
which is the linear interpolation of $K$ on $[t_{i-1},t_i]$. Since $K$ is
piecewise linear with respect to the partition, the representation holds for
all $u\in[0,1]$.
\end{proof}

Each ramp $H_k$ is a local geometric activation on the probability interval
$[t_{k-1},t_k]$: it concentrates its entire unit deformation there and is
constant elsewhere. The coefficient $\lambda_k$ measures how much of the total
deformation of $K$ is allocated to $[t_{k-1},t_k]$. The decomposition is
illustrated in Figure~\ref{fig:ramp-dictionary}.

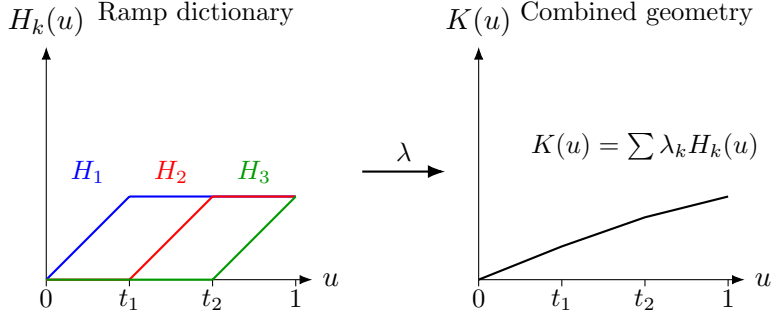
\begin{figure}[t]
\centering
\begin{tikzpicture}[scale=1.1, >=Latex]
\draw[->] (0,0) -- (3.2,0) node[right] {$u$};
\draw[->] (0,0) -- (0,2.8) node[above] {$H_k(u)$};
\node at (1.8,3.2) {\small Ramp dictionary};
\foreach \x/\label in {0/0,1/ t_1,2/ t_2,3/1}
{
    \draw (\x,0) -- (\x,-0.08);
    \node[below] at (\x,0) {\small $\label$};
}
\draw[thick,blue] (0,0) -- (1,1) (1,1) -- (3,1);
\draw[thick,red] (0,0) -- (1,0) (1,0) -- (2,1) (2,1) -- (3,1);
\draw[thick,green!60!black] (0,0) -- (2,0) (2,0) -- (3,1);
\node[blue] at (0.5,1.3) {\small $H_1$};
\node[red] at (1.5,1.3) {\small $H_2$};
\node[green!60!black] at (2.5,1.3) {\small $H_3$};
\draw[->, thick] (3.8,1.3) -- (4.8,1.3) node[midway, above] {\small $\lambda$};
\draw[->] (5.2,0) -- (8.4,0) node[right] {$u$};
\draw[->] (5.2,0) -- (5.2,2.8) node[above] {$K(u)$};
\node at (7.1,3.2) {\small Combined geometry};
\foreach \x/\label in {5.2/0,6.2/ t_1,7.2/ t_2,8.2/1}
{
    \draw (\x,0) -- (\x,-0.08);
    \node[below] at (\x,0) {\small $\label$};
}
\draw[thick,black] (5.2,0) -- (6.2,0.4) -- (7.2,0.75) -- (8.2,1);
\node at (7.2,1.6) {\small $K(u)=\sum \lambda_k H_k(u)$};
\end{tikzpicture}
\caption{Ramp dictionary representation of probability geometries. Each
elementary ramp $H_k$ activates on one probability interval and remains
constant thereafter. A deformation map $K$ is obtained as a convex combination
of the ramps, with coefficients $\lambda_k$ determining where the geometry
concentrates its deformation.}
\label{fig:ramp-dictionary}
\end{figure}

\subsection{Admissibility and the interior of the simplex}

The individual ramps are \emph{not} admissible: each $H_k$ is constant outside
its activation interval, so no vertex of the simplex generates a chart. The
admissible geometries fill the interior.

\begin{proposition}[Admissible geometries are the relative interior]
\label{prop:interior}
Let $\mathcal H_n$ be the ramp dictionary of a partition
$0=t_0<\dots<t_n=1$ and let $K_\lambda=\sum_k\lambda_kH_k$,
$\lambda\in\Delta^{n-1}$. Then $K_\lambda\in\mathcal K^\circ$ if and only if
$\lambda_k>0$ for every $k$. Consequently, the admissible piecewise-linear
deformations subordinate to the partition correspond exactly to the relative
interior of the simplex $\Delta^{n-1}$.
\end{proposition}

\begin{proof}
On $(t_{k-1},t_k)$ the map $K_\lambda$ is affine with slope
$\lambda_k/(t_k-t_{k-1})$. If every $\lambda_k>0$, all slopes are positive and
$K_\lambda$ is strictly increasing. If $\lambda_k=0$ for some $k$, then
$K_\lambda$ is constant on $[t_{k-1},t_k]$ and strict monotonicity fails. The
last claim follows from Theorem~\ref{thm:ramp}, under which
$\lambda\mapsto K_\lambda$ identifies $\Delta^{n-1}$ with the monotone
piecewise-linear deformations on the partition.
\end{proof}

\begin{remark}[Boundary allocations as degenerate geometries]
Allocations on the boundary of the simplex are still monotone deformation maps
and generate monotone non-decreasing images $K_\lambda\circ F$, but these fail
strict monotonicity and are not charts: the corresponding ``geometry''
collapses entire probability intervals to single points, and the pullback
$G_\lambda^{-1}$, hence the barycenter, is undefined. The inferential theory of
\cite{Cojocea2026b} requires charts; accordingly, in everything that follows,
statements about barycenters implicitly restrict to interior allocations. This
is the precise sense in which the strictness requirement of
Definition~\ref{def:deformation} is not a technicality.
\end{remark}

\subsection{Geometric intensity}

The ramp coefficients admit an intensity reading which will guide the examples.
When $K$ is differentiable, the derivative $K'(u)$ measures the local rate at
which deformation accumulates at probability level $u$: we call it the
\emph{geometric intensity} of the deformation. Regions where $K'$ is large
receive magnified coordinate resolution; regions where $K'$ is small are
compressed. See Figure~\ref{fig:geometric-intensity}.

More generally, any $K\in\mathcal K$ defines a probability measure $dK$ on
$[0,1]$ through its increments, and
\[
K(u)=\int_0^u dK(s),
\]
so a probability geometry is completely described by how this
\emph{deformation measure} distributes mass across probability space. In the
piecewise-linear case the ramp weights are exactly the masses
\[
\lambda_k=dK\bigl([t_{k-1},t_k]\bigr),
\]
and Theorem~\ref{thm:ramp} is the statement that the geometry is determined by
a discretization of its intensity. The identity deformation $K(u)=u$ has
uniform intensity: it is the canonical geometry, under which $G=F$ and the
probability coordinates are undeformed.

\begin{figure}[t]
\centering
\begin{tikzpicture}[scale=1.05, >=Latex]
\draw[->, thick] (0,0) -- (4.5,0) node[right] {$u$};
\draw[->, thick] (0,0) -- (0,3.6) node[above] {$K(u)$};
\node[font=\small] at (2,4.1) {\textbf{Probability geometry}};
\foreach \x/\label in {0/0,1.33/{t_1},2.66/{t_2},4/{1}}
{
    \draw (\x,0) -- (\x,-0.08);
    \node[below] at (\x,0) {\small $\label$};
}
\fill[blue!8] (0,0) rectangle (1.33,3.3);
\fill[orange!10] (1.33,0) rectangle (2.66,3.3);
\fill[red!8] (2.66,0) rectangle (4,3.3);
\draw[line width=1.8pt, color=violet]
(0,0) -- (1.33,0.45) -- (2.66,1.1) -- (4,3);
\filldraw[violet] (0,0) circle (1.2pt);
\filldraw[violet] (1.33,0.45) circle (1.2pt);
\filldraw[violet] (2.66,1.1) circle (1.2pt);
\filldraw[violet] (4,3) circle (1.2pt);
\node[violet] at (2.6,2.6) {\small $K(u)$};
\node[blue!70!black] at (0.65,3.0) {\small weak};
\node[orange!80!black] at (2.0,3.0) {\small moderate};
\node[red!70!black] at (3.3,3.0) {\small strong};
\draw[->, line width=1.3pt] (5.0,1.7) -- (6.2,1.7);
\node at (5.6,2.15) {\small derivative};
\node at (5.6,1.3) {\small geometric intensity};
\draw[->, thick] (7,0) -- (11.5,0) node[right] {$u$};
\draw[->, thick] (7,0) -- (7,3.6) node[above] {$K'(u)$};
\node[font=\small] at (9.2,4.1) {\textbf{Geometric intensity}};
\foreach \x/\label in {7/0,8.33/{t_1},9.66/{t_2},11/{1}}
{
    \draw (\x,0) -- (\x,-0.08);
    \node[below] at (\x,0) {\small $\label$};
}
\fill[blue!8] (7,0) rectangle (8.33,3.3);
\fill[orange!10] (8.33,0) rectangle (9.66,3.3);
\fill[red!8] (9.66,0) rectangle (11,3.3);
\draw[line width=2pt, color=blue!70!black] (7,0.35) -- (8.33,0.35);
\draw[line width=2pt, color=orange!90!black] (8.33,0.55) -- (9.66,0.55);
\draw[line width=2pt, color=red!75!black] (9.66,1.8) -- (11,1.8);
\draw[dashed,gray] (8.33,0) -- (8.33,1.8);
\draw[dashed,gray] (9.66,0) -- (9.66,1.8);
\node[blue!70!black] at (7.65,0.75) {\small low intensity};
\node[orange!90!black] at (9.0,0.95) {\small medium};
\node[red!75!black] at (10.35,2.1) {\small high intensity};
\end{tikzpicture}
\caption{Geometric intensity associated with a deformation map. The map $K$
redistributes coordinate resolution across probability space; its derivative
$K'$ measures the local intensity of the redistribution. Large values of $K'$
mark probability regions where the geometry concentrates.}
\label{fig:geometric-intensity}
\end{figure}
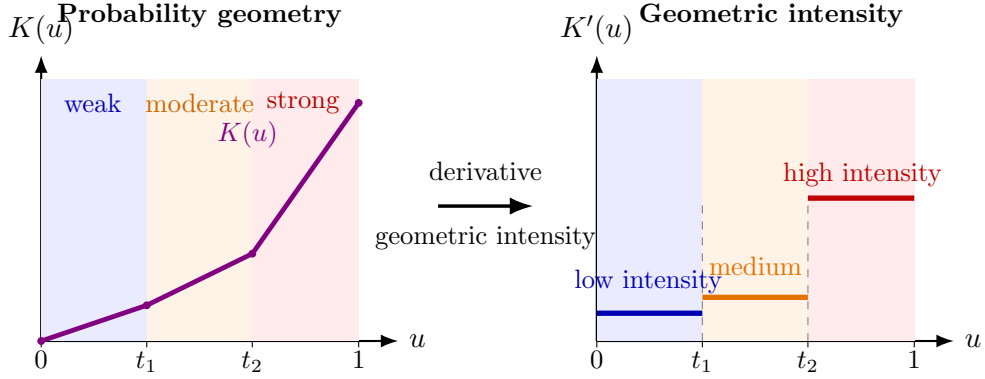

\section{Approximation of admissible deformations}
\label{sec:approx}

Throughout this section $\dinf(K,\widetilde K)=\sup_{u\in[0,1]}
|K(u)-\widetilde K(u)|$ denotes the uniform distance.

\subsection{Density of piecewise-linear geometries}

\begin{proposition}[Density]\label{prop:density}
Let $K\in\mathcal K$ and let $(\pi_n)$ be a sequence of partitions of $[0,1]$
with mesh tending to zero. Let $K^{(n)}$ be the piecewise-linear interpolant of
$K$ at the nodes of $\pi_n$. Then $K^{(n)}\in\mathcal K$, and
\[
\dinf\bigl(K^{(n)},K\bigr)\longrightarrow0 .
\]
If moreover $K\in\mathcal K^\circ$, then every $K^{(n)}\in\mathcal K^\circ$.
\end{proposition}

\begin{proof}
Interpolation preserves monotonicity, endpoint values, and continuity, so
$K^{(n)}\in\mathcal K$; if $K$ is strictly increasing its nodal values are
strictly increasing, hence so is the interpolant. For the uniform convergence,
let $\omega_K$ be the modulus of continuity of $K$ (finite and vanishing at $0$
since $K$ is continuous on a compact interval). On each cell $[t_{i-1},t_i]$ of
$\pi_n$, both $K$ and $K^{(n)}$ take values in
$[K(t_{i-1}),K(t_i)]$---for $K$ by monotonicity, for $K^{(n)}$ because it
interpolates the endpoint values linearly---so
$|K^{(n)}(u)-K(u)|\le K(t_i)-K(t_{i-1})\le\omega_K(|\pi_n|)$, where $|\pi_n|$
is the mesh. Hence $\dinf(K^{(n)},K)\le\omega_K(|\pi_n|)\to0$.
\end{proof}

Combining Proposition~\ref{prop:density} with Theorem~\ref{thm:ramp}: every
admissible probability geometry is a uniform limit of geometries generated by
finite ramp dictionaries with interior allocation vectors. The
infinite-dimensional class of charts on an interval is thus approximated by
finite-dimensional convex simplices, at the price of a partition.

\begin{remark}[Rate for smooth deformations]\label{rem:rate}
If $K\in C^2[0,1]$, the classical interpolation bound gives
\[
\dinf\bigl(K^{(n)},K\bigr)\le\tfrac18\,\|K''\|_\infty\,|\pi_n|^2 ,
\]
so ramp interpolation on $n$ uniform cells converges at rate $n^{-2}$. This
rate is visible in the numerical study of Section~\ref{sec:numerics}.
\end{remark}

\subsection{Bernstein approximation: a smooth ramp dictionary}

Piecewise-linear geometries are not differentiable at the knots. A smooth
constructive scheme is provided by the Bernstein operator
\[
(B_nK)(u)=\sum_{j=0}^n K\!\left(\frac{j}{n}\right)\binom{n}{j}u^j(1-u)^{n-j},
\qquad u\in[0,1],
\]
see \cite{Lorentz1986}. Two structural properties make it exactly suited to the
present framework: it regularizes arbitrary monotone deformations into
admissible ones, and it is itself a ramp-dictionary representation---with
smooth atoms.

\begin{theorem}[Bernstein approximants of deformation maps]\label{thm:bernstein}
Let $K\in\mathcal K$ and $n\ge1$. Then:
\begin{enumerate}[label=(\alph*)]
\item $B_nK\in\mathcal K^\circ$: the Bernstein approximant of \emph{any}
monotone deformation map is an admissible deformation map;
\item $\dinf(B_nK,K)\to0$ as $n\to\infty$;
\item $B_nK$ admits the exact convex representation
\[
(B_nK)(u)=\sum_{k=1}^n\lambda_k\,J_{n,k}(u),
\qquad
\lambda_k=K\!\left(\frac{k}{n}\right)-K\!\left(\frac{k-1}{n}\right),
\]
where
\[
J_{n,k}(u)=\sum_{j=k}^n\binom{n}{j}u^j(1-u)^{n-j}
\]
is the distribution function of the $k$-th order statistic of $n$ independent
uniform variables, i.e.\ of the $\mathrm{Beta}(k,\,n-k+1)$ law. In particular
$\{J_{n,1},\dots,J_{n,n}\}\subset\mathcal K^\circ$ is a dictionary of smooth
admissible deformation maps, and the Bernstein scheme allocates to its $k$-th
atom the same weight that Theorem~\ref{thm:ramp} allocates to the ramp on
$[\tfrac{k-1}{n},\tfrac{k}{n}]$.
\end{enumerate}
\end{theorem}

\begin{proof}
(c) Since $K(0)=0$, Abel summation gives, with
$\lambda_k=K(k/n)-K((k-1)/n)$,
\[
\sum_{j=0}^nK\!\left(\frac jn\right)b_{n,j}(u)
=\sum_{j=1}^n\Bigl(\sum_{k=1}^j\lambda_k\Bigr)b_{n,j}(u)
=\sum_{k=1}^n\lambda_k\sum_{j=k}^n b_{n,j}(u)
=\sum_{k=1}^n\lambda_k J_{n,k}(u),
\]
where $b_{n,j}(u)=\binom nju^j(1-u)^{n-j}$. The identification of $J_{n,k}$ as
the $\mathrm{Beta}(k,n-k+1)$ distribution function is the classical
binomial--beta duality: for the $k$-th uniform order statistic one has
$\sum_{j\ge k}b_{n,j}(u)=\Prob(U_{(k)}\le u)$ \cite{ShorackWellner}. Each
$J_{n,k}$ is a polynomial with $J_{n,k}(0)=0$, $J_{n,k}(1)=1$, and derivative
$k\binom nk u^{k-1}(1-u)^{n-k}>0$ on $(0,1)$, hence
$J_{n,k}\in\mathcal K^\circ$.

(a) Endpoint interpolation gives $(B_nK)(0)=K(0)=0$ and $(B_nK)(1)=K(1)=1$.
By (c), $B_nK$ is a convex combination of the admissible maps $J_{n,k}$ with
weights $\lambda_k\ge0$ summing to $K(1)-K(0)=1$; at least one $\lambda_k$ is
strictly positive, so $B_nK\in\mathcal K^\circ$ by
Proposition~\ref{prop:convexity}. (Directly:
$(B_nK)'(u)=n\sum_{j=0}^{n-1}\bigl[K(\tfrac{j+1}{n})-K(\tfrac
jn)\bigr]b_{n-1,j}(u)>0$ on $(0,1)$.)

(b) is the classical uniform convergence of Bernstein polynomials for
continuous functions \cite{Lorentz1986,Rudin}.
\end{proof}

\begin{remark}[Regularization]\label{rem:regularization}
Part (a) deserves emphasis: even when $K$ has flat pieces---for instance a
boundary allocation $K_\lambda$ with some $\lambda_k=0$, which is not a
chart-generating map---its Bernstein approximant is strictly increasing and
therefore admissible. The Bernstein operator maps the full simplex into the
class of admissible geometries, repairing degeneracy while changing the map by
at most $\dinf(B_nK,K)$. For smooth $K$ the price is a slower rate: by
Voronovskaya's theorem the error is of exact order $n^{-1}$
\cite{Lorentz1986}, against $n^{-2}$ for ramp interpolation
(Remark~\ref{rem:rate}). The numerical study of Section~\ref{sec:numerics}
displays both rates.
\end{remark}

\begin{remark}[Monotone splines]\label{rem:splines}
The same construction applies to monotone spline bases: integrated M-splines
(I-splines) are non-decreasing, endpoint-normalized basis functions whose
non-negative combinations are monotone splines \cite{Ramsay1988}, and
normalized allocation vectors again produce deformation maps with the interior
corresponding to admissible geometries. Ramps are first-order I-splines; the
beta atoms $J_{n,k}$ play the analogous role in the polynomial world. Any such
system yields a finite-dimensional convex family of probability geometries
with the transfer properties of Section~\ref{sec:transfer}.
\end{remark}

\section{Transfer to probability geometry}
\label{sec:transfer}

Approximation of deformation maps would be of limited interest if it did not
control the geometric functionals built from the charts. This section shows
that it does: uniform closeness of deformations transfers, first to charts,
then to probability barycenters and initial Kolmogorov moments.

\subsection{Uniform approximation of charts}

\begin{proposition}[Chart transfer]\label{prop:chart-transfer}
Let $F$ be a chart on $I$ and $K,\widetilde K\in\mathcal K$. Then
\[
\sup_{x\in I}\,\bigl|(K\circ F)(x)-(\widetilde K\circ F)(x)\bigr|
\;\le\;
\dinf\bigl(K,\widetilde K\bigr),
\]
with equality when $F$ is surjective onto $(0,1)$, as it is for charts.
In particular, if $\dinf(K^{(n)},K)\to0$ then
$K^{(n)}\circ F\to K\circ F$ uniformly on $I$.
\end{proposition}

\begin{proof}
For $x\in I$, $|(K\circ F)(x)-(\widetilde K\circ F)(x)|
=|K(u)-\widetilde K(u)|$ with $u=F(x)\in(0,1)$; take suprema and use the
continuity of $K-\widetilde K$ on $[0,1]$ together with the density of
$F(I)=(0,1)$ in $[0,1]$.
\end{proof}

\subsection{Stability of probability barycenters}

\begin{theorem}[Barycenter stability]\label{thm:barycenter-stability}
Let $F$ be a chart on $I$, let $K\in\mathcal K^\circ$ with $G=K\circ F$, and
let $X$ be an $I$-valued random variable. Let $(K_n)\subset\mathcal K^\circ$
with $\varepsilon_n:=\dinf(K_n,K)\to0$, and set $G_n=K_n\circ F$. Then:
\begin{enumerate}[label=(\alph*)]
\item $\bigl|\E[G_n(X)]-\E[G(X)]\bigr|\le\varepsilon_n$;
\item $G\bigl(b_{G_n}(X)\bigr)\to G\bigl(b_G(X)\bigr)$, with
$\bigl|G(b_{G_n}(X))-G(b_G(X))\bigr|\le2\varepsilon_n$;
\item $b_{G_n}(X)\to b_G(X)$;
\item if in addition $G$ is differentiable on a neighbourhood of $b_G(X)$ with
derivative bounded below by $\gamma>0$ there, then for all $n$ large enough
\[
\bigl|b_{G_n}(X)-b_G(X)\bigr|\;\le\;\frac{2\varepsilon_n}{\gamma}.
\]
\end{enumerate}
\end{theorem}

\begin{proof}
(a) $|\E[G_n(X)]-\E[G(X)]|\le\E|K_n(F(X))-K(F(X))|\le\varepsilon_n$.

(b) Write $m=\E[G(X)]$, $m_n=\E[G_n(X)]$, and $b_n=b_{G_n}(X)$, so that
$G_n(b_n)=m_n$ and $G(b_G(X))=m$. Then
\[
\bigl|G(b_n)-m\bigr|
\le\bigl|G(b_n)-G_n(b_n)\bigr|+\bigl|m_n-m\bigr|
\le\varepsilon_n+\varepsilon_n=2\varepsilon_n,
\]
using Proposition~\ref{prop:chart-transfer} for the first term and (a) for the
second.

(c) By \cite{Cojocea2026a}, $m\in(0,1)$, and $G^{-1}:(0,1)\to I$ is continuous
because $G$ is a continuous strictly increasing bijection. By (b),
$G(b_n)\to m$, hence $b_n=G^{-1}(G(b_n))\to G^{-1}(m)=b_G(X)$.

(d) By (c), $b_n$ eventually lies in the neighbourhood where $G'\ge\gamma$;
the mean value theorem applied to $G$ between $b_n$ and $b_G(X)$ gives
$2\varepsilon_n\ge|G(b_n)-G(b_G(X))|\ge\gamma\,|b_n-b_G(X)|$.
\end{proof}

\begin{remark}
No differentiability, moment condition, or tail assumption on $X$ enters parts
(a)--(c): stability of the barycenter is purely a consequence of the
boundedness of probability coordinates, in line with the robustness theme of
\cite{Cojocea2026b}. Note also that the bound in (d) controls, but does not
force, the error: the actual barycenter error can be far below
$2\varepsilon_n/\gamma$ and need not decrease monotonically in $n$, since the
signed error may cross zero. Both effects are visible in
Table~\ref{tab:errors}.
\end{remark}

\subsection{Stability of initial Kolmogorov moments}

\begin{proposition}[Moment transfer]\label{prop:moment-transfer}
In the setting of Theorem~\ref{thm:barycenter-stability}, let $r\ge1$. Then
\begin{gather*}
\bigl|\E[(G_n(X))^r]-\E[(G(X))^r]\bigr|\le r\,\varepsilon_n,
\\
\bigl|G\bigl(M_r^{(G_n)}(X)\bigr)-G\bigl(M_r^{(G)}(X)\bigr)\bigr|
\le(r+1)\,\varepsilon_n,
\end{gather*}
so $M_r^{(G_n)}(X)\to M_r^{(G)}(X)$, and under the local derivative bound of
part (d),
\[
\bigl|M_r^{(G_n)}(X)-M_r^{(G)}(X)\bigr|\le\frac{(r+1)\,\varepsilon_n}{\gamma_r}
\]
for $n$ large, where $\gamma_r>0$ lower-bounds $G'$ near $M_r^{(G)}(X)$.
\end{proposition}

\begin{proof}
For $s,t\in[0,1]$, $|s^r-t^r|\le r|s-t|$; hence the first bound. The remaining
claims repeat the proof of Theorem~\ref{thm:barycenter-stability} with $m$
replaced by $\E[(G(X))^r]$ and $\varepsilon_n+r\varepsilon_n$ in place of
$2\varepsilon_n$.
\end{proof}

\subsection{Deformed intrinsic geometries and quantiles}

When the base chart is intrinsic, barycenters of deformed geometries take a
strikingly explicit form.

\begin{proposition}[Quantile representation]\label{prop:quantile}
Let $X$ have continuous strictly increasing distribution function $F=F_X$
inducing a chart on $I$, and let $K\in\mathcal K^\circ$. Then
\[
b_{K\circ F}(X)
=F^{-1}\!\bigl(p_K\bigr),
\qquad
p_K=K^{-1}\!\left(\int_0^1K(u)\,du\right).
\]
Thus the barycenter of a deformed intrinsic geometry is the quantile of $X$ at
a level $p_K\in(0,1)$ determined by the deformation map alone, independently of
the law of $X$.
\end{proposition}

\begin{proof}
By the uniformization property of intrinsic charts \cite{Cojocea2026a},
$U=F(X)\sim\mathrm{Uniform}(0,1)$, so
$\E[(K\circ F)(X)]=\E[K(U)]=\int_0^1K(u)\,du$, and
$b_{K\circ F}(X)=(K\circ F)^{-1}\bigl(\int_0^1K\bigr)
=F^{-1}\bigl(K^{-1}(\int_0^1K)\bigr)$.
\end{proof}

\begin{remark}[Reading]\label{rem:quantile-reading}
The identity deformation gives $p_K=1/2$: the intrinsic barycenter is the
median, consistently with the degeneracy analysis of \cite{Cojocea2026b},
whose collapse theorem shows that the fully empirical intrinsic barycenter is
asymptotically the sample median. A deformation reallocates the level: for the
ramp geometries of Section~\ref{sec:examples}, $p_\lambda=
K_\lambda^{-1}\bigl(\int_0^1K_\lambda\bigr)$ depends only on the allocation
vector $\lambda$, so selecting a geometric allocation amounts to selecting a
quantile level---early-concentrated allocations push the barycenter into the
lower quantiles, late-concentrated ones into the upper quantiles.
\end{remark}

\begin{remark}[Relation to Kolmogorov-moment quantiles]
\label{rem:moment-quantiles}
Proposition~\ref{prop:quantile} should be distinguished from the quantile
identities behind the estimators of \cite{CojoceaPele2026}: there, the
\emph{intrinsic initial moment} template
\[
M_r^{(F)}(X)=F^{-1}\bigl(\E[(F(X))^r]\bigr)=F^{-1}\bigl(\tfrac1{r+1}\bigr)
\]
raises the
coordinate to a power and inverts the \emph{base} chart, while here the
deformed-chart \emph{barycenter} template composes the chart with $K$ and
inverts the \emph{composition}. With the power deformation $K(u)=u^r$ the two
templates give quantile levels $\bigl(\tfrac1{r+1}\bigr)^{1/r}$ and
$\tfrac1{r+1}$ respectively. Both families are quantiles of $X$ at
distribution-free levels; the deformation calculus of this paper provides the
systematic geometry behind such identities.
\end{remark}

\section{Examples of probability geometries}
\label{sec:examples}

We illustrate the effect of geometric allocation with the standard normal
distribution, $X\sim\mathcal N(0,1)$, intrinsic chart $F=\Phi$, on the
partition
\[
0<\tfrac13<\tfrac23<1,
\qquad
\mathcal H_3=\{H_1,H_2,H_3\}.
\]
Each interior allocation $\lambda=(\lambda_1,\lambda_2,\lambda_3)$,
$\lambda_k>0$, defines the admissible deformation
$K_\lambda=\sum_k\lambda_kH_k$ and the chart
$G_\lambda=K_\lambda\circ\Phi$. All curves and reported values below were
computed deterministically in \textsf{R} to the reported precision (scripts
supplied as ancillary files); by
Proposition~\ref{prop:quantile}, the barycenters are
$b_{G_\lambda}(X)=\Phi^{-1}(p_\lambda)$ with
$p_\lambda=K_\lambda^{-1}\bigl(\int_0^1K_\lambda\bigr)$.

\subsection{Uniform allocation}

The balanced allocation $\lambda=(\tfrac13,\tfrac13,\tfrac13)$ recovers the
identity deformation: $K_\lambda(u)=u$, the canonical geometry. No probability
region is privileged, $p_\lambda=\tfrac12$, and the barycenter is the median,
\[
b_{G_\lambda}(X)=\Phi^{-1}(0.5)=0 .
\]
Figure~\ref{fig:uniform} shows the deformation map and the resulting chart,
which coincides with the intrinsic chart $\Phi$.

\begin{figure}[t]
\centering
\includegraphics[width=\textwidth]{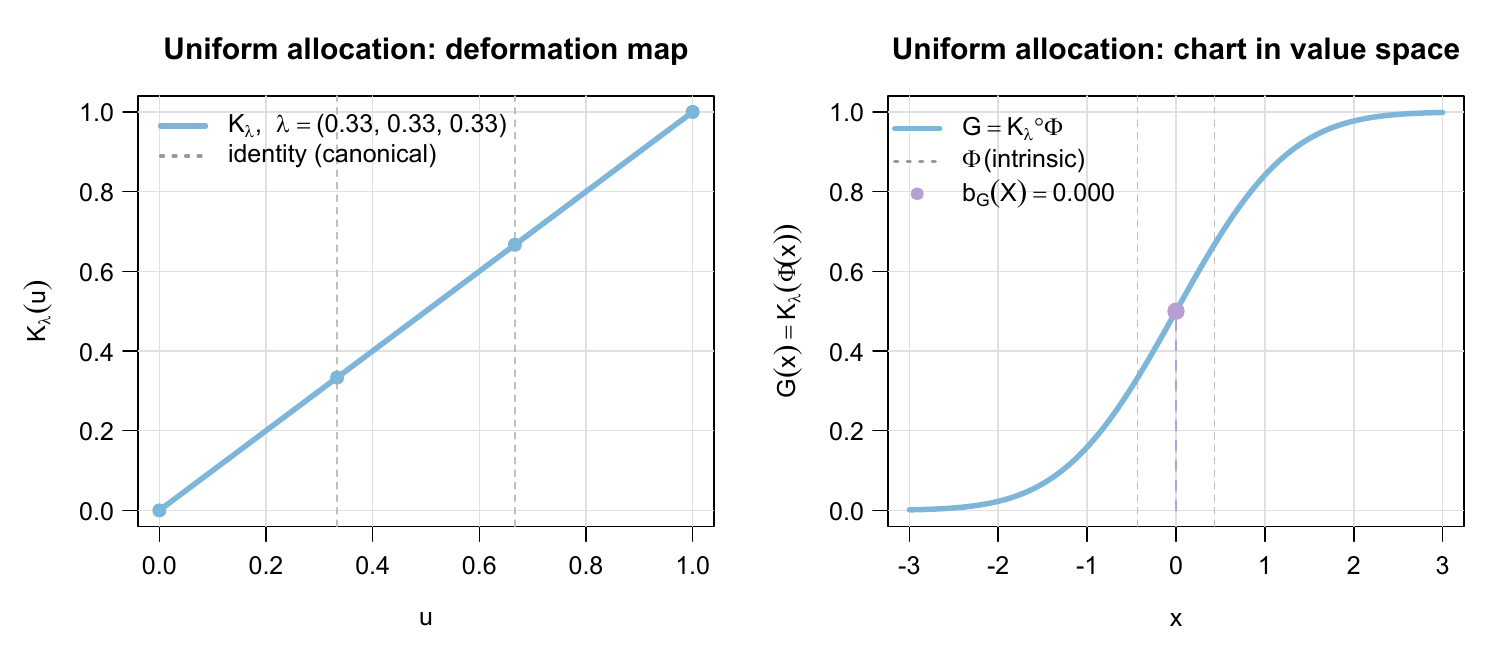}
\caption{Uniform allocation $\lambda=(1/3,1/3,1/3)$ on the partition
$(1/3,2/3)$. Left: the deformation map is the identity (canonical geometry).
Right: the chart $G_\lambda=K_\lambda\circ\Phi$ coincides with the intrinsic
chart; the barycenter (marker) sits at the median, $b_{G_\lambda}(X)=0$.
Exact curves.}
\label{fig:uniform}
\end{figure}

\subsection{Early-concentrated allocation}

The allocation $\lambda=(0.65,0.25,0.10)$ concentrates most of the deformation
on the lower probability region $[0,\tfrac13]$: the geometric intensity is
$1.95$ there, against $0.30$ on the upper region. The exact computation gives
\[
\int_0^1K_\lambda=0.68333,
\qquad
p_\lambda=0.37778,
\qquad
b_{G_\lambda}(X)=\Phi^{-1}(0.37778)=-0.3113 .
\]
The geometry magnifies resolution at small probability levels and,
consistently with Remark~\ref{rem:quantile-reading}, pulls the barycenter into
the lower quantiles of $X$. See Figure~\ref{fig:early}: in value space the
chart rises steeply through the left region and saturates early.

\begin{figure}[t]
\centering
\includegraphics[width=\textwidth]{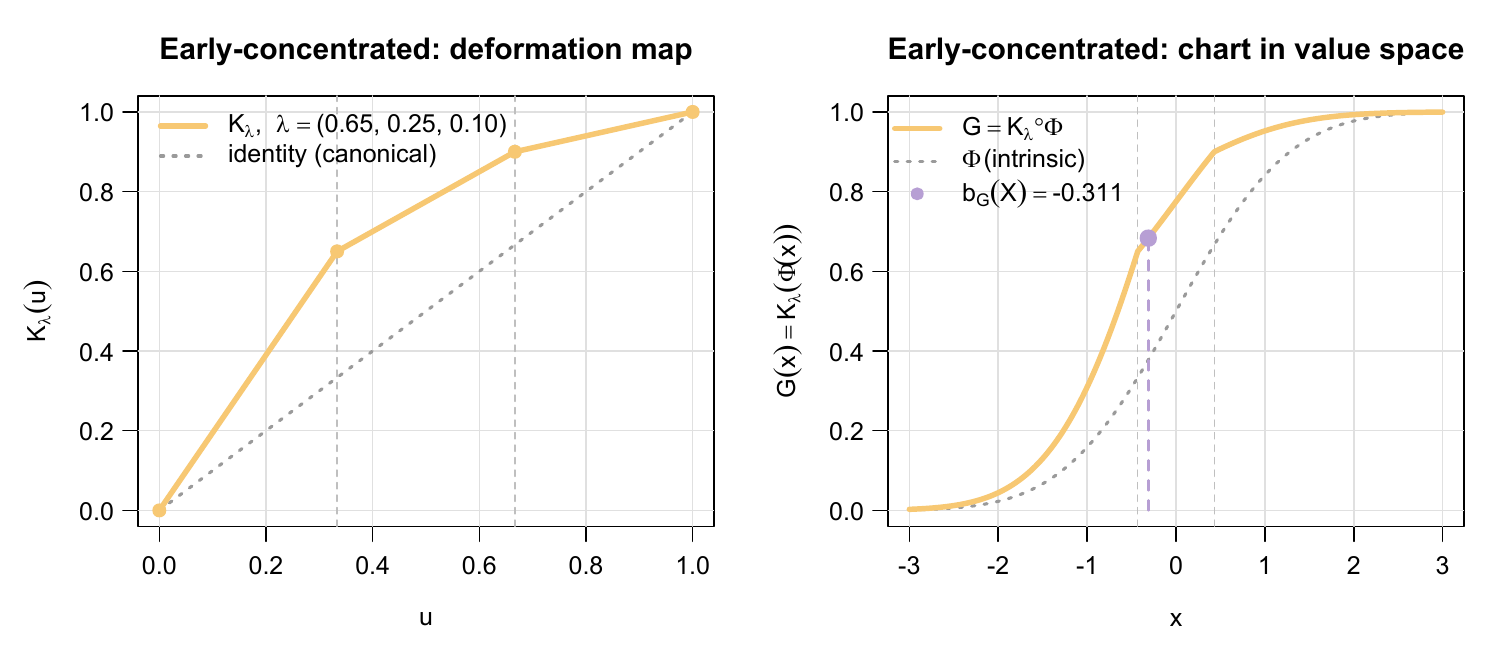}
\caption{Early-concentrated allocation $\lambda=(0.65,0.25,0.10)$. Left:
deformation map against the identity; the kinks sit at the partition points.
Right: the chart $G_\lambda=K_\lambda\circ\Phi$ against the intrinsic chart
$\Phi$; the kinks in value space sit at $\Phi^{-1}(1/3)=-0.431$ and
$\Phi^{-1}(2/3)=0.431$, and the barycenter (marker) is
$b_{G_\lambda}(X)=-0.311$. Exact curves.}
\label{fig:early}
\end{figure}

\subsection{Late-concentrated allocation}

The mirror allocation $\lambda=(0.10,0.25,0.65)$ concentrates deformation near
the upper probability region, with
\[
\int_0^1K_\lambda=0.31667,
\qquad
p_\lambda=0.62222,
\qquad
b_{G_\lambda}(X)=\Phi^{-1}(0.62222)=+0.3113 .
\]
By symmetry of the two allocations about the exchange
$\lambda\mapsto(\lambda_3,\lambda_2,\lambda_1)$ and of the standard normal
about $0$, the barycenter is the exact reflection of the previous one. See
Figure~\ref{fig:late}.

\begin{figure}[t]
\centering
\includegraphics[width=\textwidth]{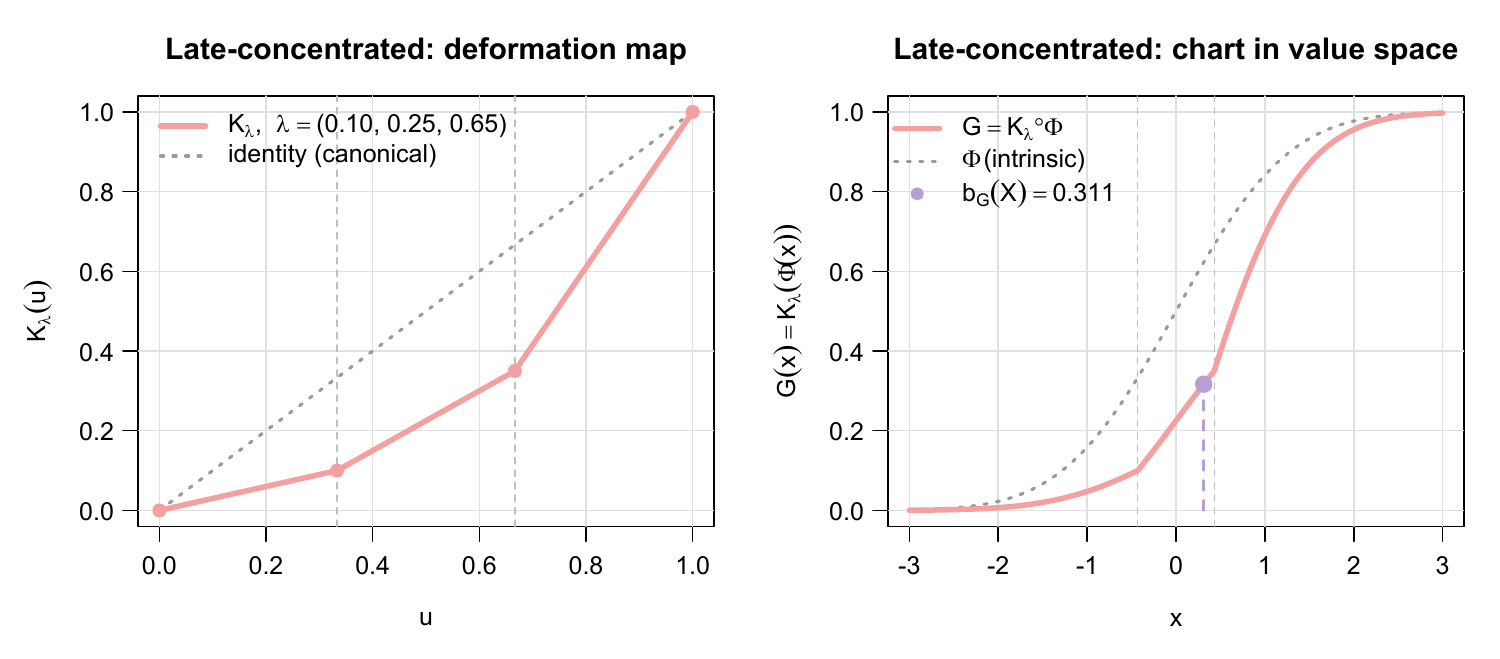}
\caption{Late-concentrated allocation $\lambda=(0.10,0.25,0.65)$. The
geometry allocates most of its deformation to high probability levels,
producing increased resolution in the upper part of the probability scale and
the barycenter $b_{G_\lambda}(X)=+0.311$. Exact curves.}
\label{fig:late}
\end{figure}

\section{Numerical illustration: constructive schemes at work}
\label{sec:numerics}

We quantify the two constructive schemes of Section~\ref{sec:approx} on a
smooth admissible target: the exponential tilt
\[
K_\alpha(u)=\frac{e^{\alpha u}-1}{e^{\alpha}-1},
\qquad \alpha=4 .
\]
This deformation is the distribution function of the canonical exponential
family on $(0,1)$ studied in \cite{Cojocea2026c}, so the
experiment simultaneously illustrates how the geometries arising there are
captured by finite convex dictionaries. For each $n\in\{4,8,16,32,64,128\}$ we
compute, on the uniform partition with $n$ cells, the ramp interpolant
$K^{(n)}$ (Proposition~\ref{prop:density}) and the Bernstein approximant
$B_nK_\alpha$ (Theorem~\ref{thm:bernstein}), their uniform errors, and
the induced barycenter errors $|b_{G_n}(X)-b_G(X)|$ for deformed intrinsic
geometries of a standard normal and a standard Cauchy variable. All quantities
are computed deterministically in \textsf{R} to the precision reported (scripts
supplied as ancillary files); no Monte Carlo variability is involved. The
reference values are
\begin{gather*}
\int_0^1K_4=0.23134,
\qquad
p_{K_4}=0.64881,
\\
b^{\mathcal N}=\Phi^{-1}(p_{K_4})=0.38210,
\qquad
b^{\mathcal C}=0.50481 .
\end{gather*}

\begin{table}[t]
\centering
\caption{Approximation errors for the exponential-tilt deformation
$K_\alpha$, $\alpha=4$: uniform errors of the ramp interpolant and the
Bernstein approximant on $n$ uniform cells, and induced errors on the
probability barycenters of deformed intrinsic geometries (standard normal and
standard Cauchy).}
\label{tab:errors}
\small
\setlength{\tabcolsep}{4.5pt}
\begin{tabular}{r cc cc cc}
\toprule
& \multicolumn{2}{c}{$\dinf(\cdot,K_\alpha)$} & \multicolumn{2}{c}{$|\Delta b_G|$, Gaussian} & \multicolumn{2}{c}{$|\Delta b_G|$, Cauchy} \\
\cmidrule(lr){2-3}\cmidrule(lr){4-5}\cmidrule(lr){6-7}
$n$ & ramp & Bernstein & ramp & Bernstein & ramp & Bernstein \\
\midrule
4 & $7.94\cdot10^{-2}$ & $1.37\cdot10^{-1}$ & $2.37\cdot10^{-2}$ & $1.22\cdot10^{-1}$ & $3.43\cdot10^{-2}$ & $1.71\cdot10^{-1}$ \\
8 & $2.50\cdot10^{-2}$ & $7.12\cdot10^{-2}$ & $4.27\cdot10^{-4}$ & $6.24\cdot10^{-2}$ & $6.24\cdot10^{-4}$ & $8.91\cdot10^{-2}$ \\
16 & $7.04\cdot10^{-3}$ & $3.62\cdot10^{-2}$ & $1.52\cdot10^{-3}$ & $3.14\cdot10^{-2}$ & $2.22\cdot10^{-3}$ & $4.54\cdot10^{-2}$ \\
32 & $1.87\cdot10^{-3}$ & $1.82\cdot10^{-2}$ & $5.98\cdot10^{-5}$ & $1.57\cdot10^{-2}$ & $8.75\cdot10^{-5}$ & $2.29\cdot10^{-2}$ \\
64 & $4.82\cdot10^{-4}$ & $9.13\cdot10^{-3}$ & $1.09\cdot10^{-4}$ & $7.88\cdot10^{-3}$ & $1.59\cdot10^{-4}$ & $1.15\cdot10^{-2}$ \\
128 & $1.22\cdot10^{-4}$ & $4.57\cdot10^{-3}$ & $3.94\cdot10^{-5}$ & $3.94\cdot10^{-3}$ & $5.76\cdot10^{-5}$ & $5.75\cdot10^{-3}$ \\
\bottomrule
\end{tabular}
\end{table}

\begin{figure}[t]
\centering
\includegraphics[width=\textwidth]{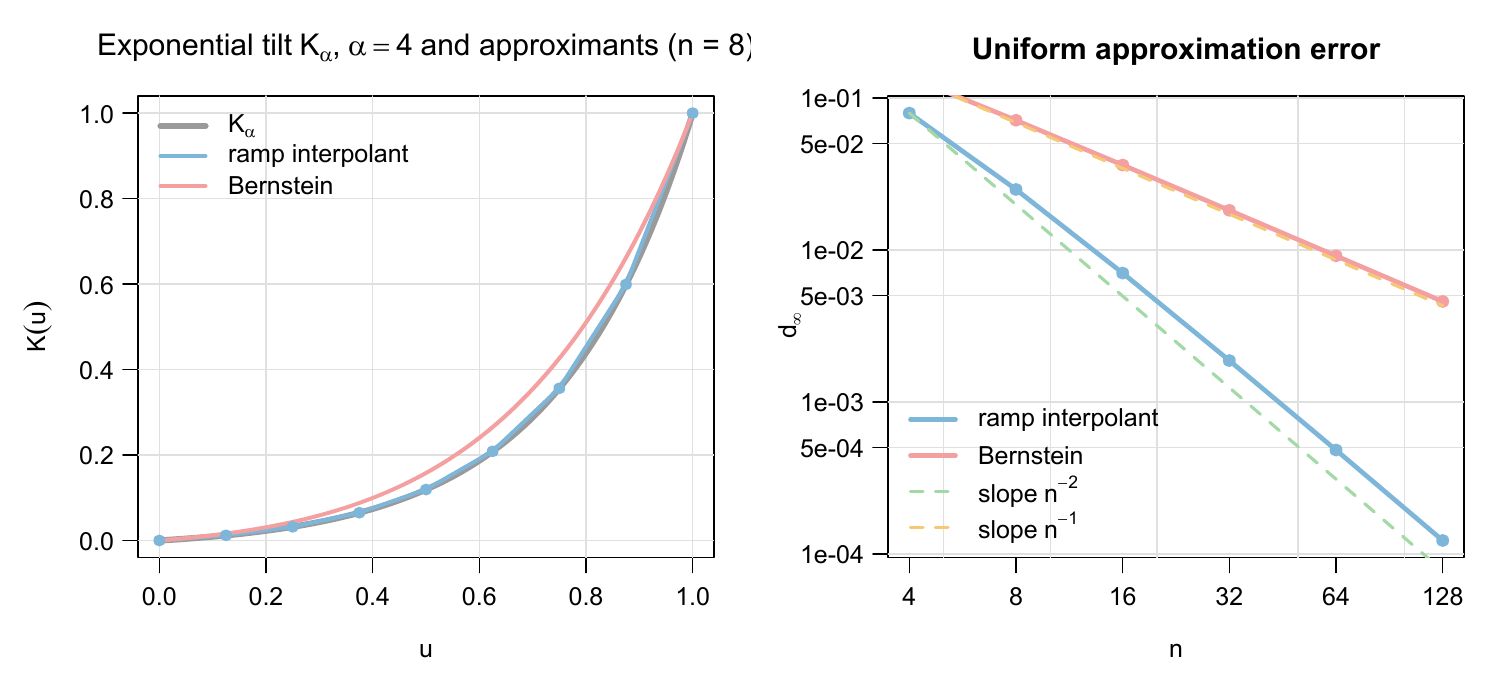}
\caption{Left: the exponential-tilt deformation $K_\alpha$, $\alpha=4$, with
its ramp interpolant and Bernstein approximant at $n=8$. Right: uniform
errors against $n$ (log--log), with reference slopes $n^{-2}$ (ramp,
Remark~\ref{rem:rate}) and $n^{-1}$ (Bernstein,
Remark~\ref{rem:regularization}).}
\label{fig:convergence}
\end{figure}

Three observations, all consistent with the theory.
First, the uniform errors decay at the predicted rates: ramp interpolation at
$n^{-2}$ (the error falls by a factor of about $650$ from $n=4$ to $n=128$,
against $1024$ for an exact quadratic rate on this pre-asymptotic range), the
Bernstein scheme at $n^{-1}$ (factor $30$, against $32$ exactly); see
Figure~\ref{fig:convergence}. Second, the induced barycenter errors respect,
and generally undershoot by a wide margin, the bound
$2\varepsilon_n/\gamma$ of Theorem~\ref{thm:barycenter-stability}(d): for the
ramp scheme at $n=8$ the barycenter error is already $4\cdot10^{-4}$, three
orders of magnitude below the uniform error would suggest, because the signed
error of the interpolant nearly integrates to zero across cells. The
non-monotonicity of the ramp barycenter-error column (compare $n=8$ and
$n=16$) illustrates that the transfer bound controls the error without forcing
monotone decay. Third, the Cauchy barycenter errors exceed the Gaussian ones
by roughly the ratio of quantile densities at the target level, as the
mean-value factor $1/\gamma$ predicts; heavy tails of $X$ affect the constant,
not the rate---in line with the boundedness philosophy of the framework.

\section{Discussion and outlook}
\label{sec:discussion}

The results of this paper reduce the infinite-dimensional selection of a
probability geometry to the selection of an allocation vector in a simplex.
The deformation map does not transform values: it reallocates geometric
intensity across intrinsic probability coordinates, and the ramp weights
$\lambda_k$---the masses of the deformation measure $dK$---describe this
allocation exactly. Monotone piecewise-linear geometries are convex
combinations of local ramp activations (Theorem~\ref{thm:ramp}), admissibility
carves out the interior of the simplex (Proposition~\ref{prop:interior}),
density holds in the admissible class (Proposition~\ref{prop:density}), and
the Bernstein scheme shows that the same convex calculus runs with smooth
order-statistic atoms while automatically restoring admissibility
(Theorem~\ref{thm:bernstein}). The transfer results of
Section~\ref{sec:transfer} guarantee that these approximations are faithful
for the associated geometric functionals: barycenters and Kolmogorov
moments move continuously, with explicit constants, under uniform
perturbations of the geometry.

Several directions follow naturally.

\begin{itemize}
\item \emph{Geometry selection and estimation.} The simplex parametrisation
turns chart selection, the open problem singled out in
\cite{Cojocea2026b}, into finite-dimensional estimation: a family
$\{K_\lambda\circ F_\theta\}$ combines a parametric base chart with a
low-dimensional geometric allocation, and the two-step calibration theory of
\cite{Cojocea2026b} applies conditionally on $\lambda$. Optimizing $\lambda$
against an inferential criterion (asymptotic variance, gross-error
sensitivity) is a convex-parametrised problem left for future work.
\item \emph{Sparse and adaptive allocations.} Sparse vectors $\lambda$
generate geometries concentrated on few probability regions; adaptive schemes
may let the data place the intensity where the law of $X$ demands resolution.
The quantile representation (Proposition~\ref{prop:quantile}) makes the
statistical meaning of such choices transparent for intrinsic bases.
\item \emph{Transported structures.} The $G$-exponential families of
\cite{Cojocea2026c} and the stable-law quantile estimators of
\cite{CojoceaPele2026} are both indexed by charts; replacing their charts by
dictionary geometries $G_\lambda$ yields finite-dimensional families of
transported models whose behaviour under the transfer bounds of
Section~\ref{sec:transfer} is controlled.
\item \emph{Beyond one dimension.} Componentwise deformations of copula-based
multivariate coordinates, and genuinely joint deformations of $[0,1]^d$,
remain open, alongside the multivariate questions raised in
\cite{Cojocea2026a,Cojocea2026b}.
\end{itemize}

In this geometric reading, probability space carries many
geometries at once: the intrinsic coordinates provide the canonical reference,
and deformation maps redistribute resolution across it. The contribution of
this paper is to show that this multiplicity is not unmanageable---it is, up
to uniform approximation with explicit geometric consequences, the interior of
a simplex.


\end{document}